\documentclass[11pt]{article}
\pdfoutput=1
\usepackage{enumerate}
\usepackage{pdfsync}
\usepackage[OT1]{fontenc}

\usepackage[usenames]{color}
\usepackage{smile}
\usepackage[colorlinks,
            linkcolor=red,
            anchorcolor=blue,
            citecolor=blue
            ]{hyperref}
\usepackage{fullpage}
\usepackage{hyperref}
\usepackage[protrusion=true,expansion=true]{microtype}
\usepackage{pbox}
\usepackage{setspace}
\usepackage{tabularx}
\usepackage{float}
\usepackage{wrapfig,lipsum}
\usepackage{enumitem}

\makeatletter
\newcommand*{\rom}[1]{\expandafter\@slowromancap\romannumeral #1@}
\makeatother

\def \fc {f_{\cN\text{c}}}
\def \fsc {f_{\cN\text{sc}}}
\def \fnc {f_{\cC}}

\newcommand{\la}{\langle}
\newcommand{\ra}{\rangle}
\def \dist {\text{dist}}
\def \lg {c_{\gamma}}

\begin{document}

\title{\huge Lower Bounds for Smooth Nonconvex Finite-Sum Optimization}
\author
{
	Dongruo Zhou\thanks{Department of Computer Science, University of California, Los Angeles, CA 90095, USA; e-mail: {\tt drzhou@cs.ucla.edu}} 
	~~~and~~~
	Quanquan Gu\thanks{Department of Computer Science, University of California, Los Angeles, CA 90095, USA; e-mail: {\tt qgu@cs.ucla.edu}}
}
\date{January 23, 2019}
\maketitle

\begin{abstract}
Smooth finite-sum optimization has been widely studied in both convex and nonconvex settings. However, existing lower bounds for finite-sum optimization are mostly limited to the setting where each component function is (strongly) convex, while the lower bounds for nonconvex finite-sum optimization remain largely unsolved. In this paper, we study the lower bounds for smooth nonconvex finite-sum optimization, where the objective function is the average of $n$ nonconvex component functions. We prove tight lower bounds for the complexity of finding $\epsilon$-suboptimal point and $\epsilon$-approximate stationary point in different settings, for a wide regime of the smallest eigenvalue of the Hessian of the objective function (or each component function). Given our lower bounds, we can show that existing algorithms including {KatyushaX} \citep{allen2018katyushax}, {Natasha} \citep{allen2017natasha}, RapGrad \citep{lan2018accelerated} and {StagewiseKatyusha} \citep{yang2018does} have achieved optimal {Incremental First-order Oracle} (IFO) complexity (i.e., number of IFO calls) up to logarithm factors for nonconvex finite-sum optimization. We also point out potential ways to further improve these complexity results, in terms of making stronger assumptions or by a different convergence analysis. 
\end{abstract}


\section{Introduction}

We consider minimizing the following unconstrained finite-sum optimization problem: 
\begin{align}
   \min_{\xb \in \RR^d} F(\xb) = \frac{1}{n}\sum_{i=1}^nf_i(\xb),\label{eq:intro_problem_we_study}
\end{align}
where each $f_i(\xb):\RR^m\rightarrow \RR$ is smooth and \emph{nonconvex} function. We are interested in the algorithmic performance of \emph{first-order algorithms} for solving \eqref{eq:intro_problem_we_study}, which have accesses to the Incremental First-order Oracle (IFO) \citep{agarwal2015lower} defined as follows:
\begin{align}
    \text{Given $\xb$ and $i\in [n]$, an IFO returns }[f_i(\xb), \nabla f_i(\xb)].\notag
\end{align}

In this paper, we consider the very general setting where $F(\xb)$ is of $(l,L)$-smoothness \citep{allen2017natasha}, 
i.e., there exist some constant $l \in \RR$ and $L>0$, such that for any $\xb, \yb \in \RR^d$,
\begin{align}\label{eq:intro_lower}
     \frac{l}{2}\|\xb - \yb\|_2^2 &\leq  F(\xb) - F(\yb) - \la \nabla F(\yb), \xb - \yb\ra\notag \\
      &\leq \frac{L}{2}\|\xb - \yb\|_2^2,
\end{align}
where $l\in \RR$ \footnote{We allow $l$ to be nonnegative, which covers the definitions of convex and strongly convex functions.} is the lower smoothness parameter, and $L>0$ is the upper smoothness parameter.
Note that conventional $L$-smoothness definition is a special case of \eqref{eq:intro_lower}, where $l = -L$.
\eqref{eq:intro_lower} is quite general, because with different choice of $l$, \eqref{eq:intro_problem_we_study} and \eqref{eq:intro_lower} together can cover various kinds of smooth finite-sum optimization problems. 
For example, when $l \geq 0$, $F(\xb)$ is convex function, and $F(\xb)$ is $\sigma$-strongly convex if $l = \sigma >0$. Such a sum-of-nonconvex optimization problem (convex functions that are average of nonconvex ones) was originally identified in \citet{shalev2015sdca}, and widely used in various machine learning problems such as principal component analysis (PCA) \citep{garber2016faster,allen2016lazysvd}. With $l\geq 0$, our goal is to find an $\epsilon$-suboptimal solution $\hat\xb$ \citep{woodworth2016tight} to \eqref{eq:intro_problem_we_study}, which satisfies
\begin{align}
    F(\hat\xb) - \inf_\xb F(\xb) \leq \epsilon.\label{suboptimal}
\end{align}
On the other hand, when $l=-\sigma < 0$, $F(\xb)$ is nonconvex, and it is called $\sigma$-almost convex \citep{carmon2018accelerated} \footnote{It is also known as $\sigma$-weakly convex \citep{yang2018does} or $\sigma$-bounded nonconvex \citet{allen2017natasha}.)}. 
It is known that finding an $\epsilon$-suboptimal solution in such nonconvex setting is NP-hard \citep{murty1987some}. Thus, our goal is instead to find an $\epsilon$-approximate stationary point $\hat\xb$ of $F(\xb)$ for general nonconvex case, which is defined as follows
\begin{align}
    \|\nabla F(\hat\xb)\|_2 \leq \epsilon.\label{station}
\end{align}

There is a vast literature on finding either \eqref{suboptimal} or \eqref{station} for \eqref{eq:intro_problem_we_study}, such as SDCA without Duality \citep{shalev2016sdca}, Natasha \citep{allen2017natasha}, KatyushaX \citep{allen2018katyushax}, RapGrad \citep{lan2018accelerated}, StagewiseKatyusha \citep{yang2018does}, RepeatSVRG \citep{agarwal2017finding, carmon2018accelerated}, to mention a few.
In specific, this line of work can be divided into two categories based on the smoothness assumption over $\{f_i(\xb)\}_{i=1}^n$. The first category of work \citep{shalev2016sdca, allen2017natasha, allen2018katyushax, agarwal2017finding, carmon2018accelerated} makes the assumption that each individual component function $f_i(\xb)$ is $L$-smooth and $F(\xb)$ is $(l,L)$ smooth. Under such an assumption, when $F(\xb)$ is convex or $\sigma$-strongly convex, SDCA without Duality and KatyushaX can find the $\epsilon$-suboptimal solution within $O(n+n^{3/4}\sqrt{L/\epsilon})$ or $O(n+n^{3/4}\sqrt{L/\sigma}\log(1/\epsilon))$ IFO calls respectively. When $F(\xb)$ is $\sigma$-almost convex, {Natasha} and RepeatSVRG can find the $\epsilon$-approximate stationary point with $O((n^{3/4}\sqrt{\sigma L} \land \sqrt{n}L)/\epsilon^2)$ IFO calls. 

The second category of work \citet{allen2017natasha, allen2018katyushax,lan2018accelerated,yang2018does} assumes that each $f_i(\xb)$ is $(-\sigma, L)$-smooth \footnote{In fact, \citet{allen2017natasha, allen2018katyushax} fall into both categories.}. With such an assumption, {RapGrad} and {StagewiseKatyusha} find $\epsilon$-approximate stationary point with $O((n\sigma + \sqrt{n\sigma L})/\epsilon^2)$ IFO calls. 

Given the above IFO complexity results, a natural research question is: 

\emph{Are these upper bounds of IFO complexity already optimal?}


We answer this question in an affirmative way by proving lower bounds on the IFO complexity for a wide regime of $l$, using carefully constructed functions.  
 More specifically, our contributions are summarized as follows:
\begin{enumerate}
    \item For the case that $F(\xb)$ is convex or $\sigma$-strongly convex (a.k.a., sum-of-nonconvex optimization), we show that without the convexity assumption on each component function $f_i(\xb)$, the lower bound of IFO complexity for any linear-span first-order randomized algorithms (See Definition \ref{eq:linear span algorithm}) to find $\epsilon$-suboptimal solution is $\Omega(n+n^{3/4}\sqrt{L/\sigma}\log(1/\epsilon))$ when $F$ is $\sigma$-strongly convex, and $\Omega(n+n^{3/4}\sqrt{L/\epsilon})$ when $F(\xb)$ is convex, where $L$ is the average smoothness parameter on $\{f_i(\xb)\}_{i=1}^n$ (See Definition \ref{def:average smooth}). That is in contrast to the lower bounds $\Omega(n+n^{1/2}\sqrt{L/\sigma}\log(1/\epsilon))$ and $\Omega(n+n^{1/2}\sqrt{L/\epsilon})$ proved by \citet{woodworth2016tight} when each component function $f_i(\xb)$ is convex. 
    \item For the case that $F(\xb)$ is $\sigma$-almost convex, we show that the lower bound of IFO complexity for any linear-span first-order randomized algorithms to find $\epsilon$-approximate stationary point is $\Omega(1/\epsilon^2(n^{3/4}\sqrt{L\sigma} \land \sqrt{n}L))$ when $\{f_i(\xb)\}_{i=1}^n$ is $L$-average smooth, and $\Omega(1/\epsilon^2(\sqrt{n L\sigma} \land L))$ when each $f_i(\xb)$ is $(-\sigma, L)$-smooth. To our best knowledge, this is the first lower bound result which precisely characterizes the dependency on the lower smoothness parameter for finding approximate stationary point.
    \item We show that many existing algorithms including SDCA without Duality \citep{shalev2016sdca}, Natasha \citep{allen2017natasha}, KatyushaX \citep{allen2018katyushax}, RapGrad \citep{lan2018accelerated}, StagewiseKatyusha \citep{yang2018does} and RepeatSVRG \citep{agarwal2017finding, carmon2018accelerated} have indeed achieved optimal IFO complexity for a large regime of the lower smoothness parameter, with slight modification of their original convergence analyses. 
\end{enumerate}

\noindent\textbf{Notation}
We define $\lg = 360$. We use $a(x) = O(b(x))$ if $a(x) \leq Cb(x)$, where $C$ is a universal constant. We use $\tilde O(\cdot)$ to hide polynomial logarithm terms. For any vector $\vb \in \RR^m$, we use $\vb_i$ to denote the $i$-th coordinate of $\vb$, and $\|\vb\|_2$ to denote its 2-norm. For any vector sequence $\{\vb^{(i)}\}_{i=1}^n$, we use $\vb^{(i)}$ to denote the $i$-th vector. We say a matrix sequence $\{\Ub^{(i)}\}_{i=1}^n \in \cO(a,b,n)$ where for each $i$, $\Ub^{(i)} \in \RR^{a\times b}$, if $\Ub^{(i)}(\Ub^{(i)})^\top = \Ib$ and $\Ub^{(i)}(\Ub^{(j)})^\top = \textbf{0}$ for any $1 \leq i \neq j \leq n$. For any sets $A,B\subseteq \RR^d$, we define the distance between them as $\dist(A,B) = \inf_{\ab \in A, \bbb \in B}\|\ab - \bbb\|_2$. For any $A \subseteq \RR^d$, we denote by $\text{Lin}\{A\}$ the linear space spanned by $a \in A$. In the rest of this paper, we use $F(\xb), f_i(\xb)$ and $F, f_i$ interchangeably when there is no confusion.


\section{Additional Related Work }\label{related_work}

In this section, we review additional related work that is not discussed in the introduction section.

\noindent\textbf{Existing lower bounds for nonconvex optimization:} To the best of our knowledge, the only existing lower bounds for nonconvex optimization are proved in \citet{carmon2017lower, carmon2017lower2, fang2018spider}. \citet{carmon2017lower, carmon2017lower2} proved the lower bounds for both deterministic and randomized algorithms on nonconvex optimization with high-order smoothness assumption. However, 
they did not consider the finite-sum structure which will bring additional dependency on the lower-smoothness parameter $l$ and the number of component functions $n$. \citet{fang2018spider} proved a lower bound for nonconvex finite-sum optimization under conventional smoothness assumption, i.e., $l=-L$. 
Our work extends this line of research, and proves matching lower bounds for nonconvex  finite-sum optimization (and sum-of-nonconvex optimization) under the refined $(l,L)$-smooth assumption.


\noindent\textbf{Existing upper bounds for first-order convex optimization:}
There existing a bunch of work focusing on establishing upper complexity bounds to find $\epsilon$-suboptimal solution for convex finite-sum optimization problems. It is well known that by treating $F(\xb)$ as a whole part, gradient descent can achieve $O(nL/\epsilon)$ IFO complexity for convex functions and $O(nL/\sigma\log(1/\epsilon))$ for $\sigma$-strongly convex functions, and accelerated gradient descent (AGD) \citep{Nesterov1983A} can achieve  $O(n\sqrt{L/\epsilon})$ IFO complexity for convex functions and $O(n\sqrt{L/\sigma}\log(1/\epsilon))$ for $\sigma$-strongly convex functions. Both IFO complexities achieved by AGD are optimal when $n = 1$ \citep{Nesterov1983A}. By using variance reduction technique \citep{roux2012stochastic,johnson2013accelerating,xiao2014proximal,defazio2014saga, mairal2015incremental, bietti2017stochastic}, the IFO complexity can be improved to be $O((n+L/\sigma)\log(1/\epsilon))$ for strongly convex functions. By combining variance reduction and Nesterov's acceleration techniques \citep{Nesterov1983A}, the IFO complexity can be further reduced to $O(n\log(1/\epsilon) + \sqrt{nL/\epsilon})$ for convex functions, and $O((n+\sqrt{nL/\sigma})\log(1/\epsilon))$ for $\sigma$-strongly convex functions \citep{allen2017katyusha}, which matches the lower bounds up to a logarithm factor.

\noindent\textbf{Existing lower bounds for first-order convex optimization:}
For deterministic optimization algorithms, it has been proved that one needs $\Omega(\sqrt{L/\epsilon})$ IFO calls for convex functions, and $\Omega(\sqrt{L/\sigma}\log(1/\epsilon))$ IFO calls for $\sigma$-strongly convex functions to find an $\epsilon$-suboptimal solution.
There is a line of work \citep{woodworth2016tight, lan2017optimal, agarwal2015lower, arjevani2016dimension} establishing the lower bounds for first-order algorithms to find $\epsilon$-suboptimal solution to the convex finite-sum optimization. More specifically, \citet{agarwal2015lower} proved a lower bound $\Omega (n+\sqrt{nL/\sigma}\log(1/\epsilon))$ for strongly convex finite-sum optimization problems, which is valid for deterministic algorithms. \citet{arjevani2016dimension} provided a dimension-free lower bound $\Omega (n+\sqrt{nL/\sigma}\log(1/\epsilon))$ for first-order algorithms with the assumption that any new iterate generated by the algorithm lies in the linear span of gradients and iterates up to the current iteration. \citet{lan2017optimal} proved a lower bound $\Omega (n+\sqrt{L/\sigma}\log(1/\epsilon))$ for a class of randomized first-order algorithms where each component function will be selected by fixed probabilities. \citet{woodworth2016tight} proved a set of lower bounds including $\Omega (n+\sqrt{L/\epsilon})$ for convex functions and $\Omega (n+\sqrt{L/\sigma}\log(1/\epsilon))$ for $\sigma$-strongly convex functions. Besides, \citet{woodworth2016tight}'s results do not need the assumption that the new iterate lies in the span of all the iterates up to the iteration, which is a more general result.

For more details on the upper bound and lower bound results, please refer to Tables \ref{table:average} and \ref{table:sepa}.

\begin{table*}[ht]
\caption{IFO Complexity comparison with the assumption that $\{f_i\}_{i=1}^n$ is average $L$-smooth and $F$ is $(l, L)$-smooth. Here $\Delta=F(\xb^{(0)}) - \inf_{\xb \in \RR^d} F(\xb)$ and $B= \dist(\xb^{(0)}, \cX^*)$, where $\cX^* = \argmin_{\xb \in \RR^d}F(\xb)$. When $l = \sigma$ or 0, the goal is to find an $\epsilon$-suboptimal solution; and when $l = -\sigma<0$, the goal is to find an $\epsilon$-approximate stationary point.}
\label{table:average}
\begin{small}
\begin{center}
\begin{tabular}{cccc}
\toprule
 $\sigma >0$ & $(\sigma,L)$ & $(0,L)$ & $(-\sigma,L)$  \\
\midrule
\multirow{3}{*}{Upper Bounds} & \multirow{2}{*}{$O\bigg(\big(n+n^{3/4}\sqrt{\frac{L}{\sigma}}\big)\log \frac{\Delta}{\epsilon}\bigg)$} & \multirow{2}{*}{$O\big(n+n^{3/4}B\sqrt{\frac{L}{\epsilon}}\big)$}  &  \multirow{2}{*}{$\tilde O\big(\frac{\Delta}{\epsilon^2}(n^{3/4}\sqrt{\sigma L}\land \sqrt{n}L)\big)$} \\
 & & & \\
 & \citep{allen2018katyushax}&\citep{allen2018katyushax} &\citep{allen2017natasha, fang2018spider} \\
\multirow{3}{*}{Lower Bounds} & \multirow{2}{*}{$\Omega\big(n+n^{3/4}\sqrt{\frac{L}{\sigma}}\log \frac{\Delta}{\epsilon}\big)$} & \multirow{2}{*}{$\Omega\big(n+n^{3/4}B\sqrt{\frac{L}{\epsilon}}\big)$}  &  \multirow{2}{*}{$\Omega\big(\frac{\Delta}{\epsilon^2}(n^{3/4}\sqrt{\sigma L}\land \sqrt{n}L)\big)$} \\
 & & & \\
 & (Theorem \ref{linear_theorem_exp_sc})& (Theorem \ref{linear_theorem_exp_c})&(Theorem \ref{linear_theorem_exp_nc}) \\
\bottomrule
\end{tabular}
\end{center}
\end{small}
\end{table*}

\begin{table*}[ht]
\caption{IFO Complexity comparison with the assumption that each $f_i$ is $(l, L)$-smooth. Here $\Delta=F(\xb^{(0)}) - \inf_{\xb \in \RR^d} F(\xb)$ and $B=\dist(\xb^{(0)}, \cX^*)$ where $\cX^* = \argmin_{\xb \in \RR^d}F(\xb)$. When $l = \sigma$ or 0, the goal is to find an $\epsilon$-suboptimal solution; and when $l = -\sigma<0$, the goal is to find an $\epsilon$-approximate stationary point.}
\label{table:sepa}
\begin{small}
\begin{center}
\begin{tabular}{cccc}
\toprule
 $\sigma >0$ & $(\sigma,L)$ & $(0,L)$ & $(-\sigma,L)$  \\
\midrule
\multirow{3}{*}{Upper Bounds} & \multirow{2}{*}{$O\bigg(\big(n+\sqrt{\frac{nL}{\sigma}}\big)\log \frac{\Delta}{\epsilon}\bigg)$} & \multirow{2}{*}{$O\big(n+B\sqrt{\frac{nL}{\epsilon}}\big)$}  &  \multirow{2}{*}{$\tilde O\big(\frac{\Delta}{\epsilon^2}(n\sigma + \sqrt{n\sigma L})\land \sqrt{n}L)\big)$} \\
 & & & \\
 & \multirow{2}{*}{\citep{allen2017katyusha}}&\multirow{2}{*}{\citep{allen2017katyusha}} &\citep{lan2018accelerated} \\
 &&&\citep{fang2018spider}\\
\multirow{3}{*}{Lower Bounds} & \multirow{2}{*}{$\Omega\big(n+\sqrt{\frac{nL}{\sigma}}\log \frac{\Delta}{\epsilon}\big)$} & \multirow{2}{*}{$\Omega\big(n+B\sqrt{\frac{nL}{\epsilon}}\big)$}  &  \multirow{2}{*}{$\Omega\big(\frac{\Delta}{\epsilon^2}(\sqrt{n\sigma L}\land L)\big)$} \\
 & & & \\
 & \citep{woodworth2016tight}& \citep{woodworth2016tight}&(Theorem \ref{linear_theorem_sep_nc}) \\
\bottomrule
\end{tabular}
\end{center}
\end{small}
\end{table*}
\section{Preliminaries}
We first present the formal definitions of $(l,L)$-smoothness and average smoothness, which will be used throughout the proof.
\begin{definition}
For any differentiable function $f: \RR^m \rightarrow \RR$, we say $f$ is $(l, L)$-smooth for some $l\in \RR$ and $L\in \RR^+$  if for any $\xb, \yb \in \RR^m$, it holds that
\begin{align*}
    \frac{l}{2}\|\xb - \yb\|_2^2 &\leq  f(\xb) - f(\yb) - \la \nabla f(\yb), \xb - \yb\ra\notag \\
      &\leq \frac{L}{2}\|\xb - \yb\|_2^2.
\end{align*}
We denote such a function class by $\cS^{(l,L)}$. In particular, we say $f$ is $L$-smooth if $f \in \cS^{(-L,L)}$. 
\end{definition}
Note that if $f$ is twice differentiable, then $f \in \cS^{(l, L)}$ if and only if
$l\Ib \preceq \nabla^2 f(\xb) \preceq L\Ib$ for any $\xb \in \RR^m$.

\begin{definition}\label{def:average smooth}
For any differentiable functions $\{f_i\}_{i=1}^n: \RR^m \rightarrow \RR$, we say $\{f_i\}_{i=1}^n$ is $L$-average smooth for some $L>0$ if $\EE_i \|\nabla f_i(\xb) - \nabla f_i(\yb)\|_2^2 \leq L^2 \|\xb - \yb\|_2^2$ for any $\xb, \yb \in \RR^m$, where $\EE_i X(i) = 1/n\cdot\sum_{i=1}^n X(i)$ for any random variable $X(i)$. We denote such a function class by $\cV^{(L)}$.
\end{definition}
It is worth noting that if $\{f_i\}$ satisfy that for each $i$, $f_i \in \cS^{(-L,L)}$, then $\{f_i\} \in \cV^{(L)}$. 

In this work, we focus on the \emph{linear-span randomized first-order algorithm}, which is defined as follows:

\begin{definition}\label{eq:linear span algorithm}
Given an initial point $\xb^{(0)}$, a \emph{linear-span randomized first-order algorithm} $\cA$ is defined as a measurable mapping from functions $\{f_i\}_{i=1}^n$ to an infinite sequence of point and index pairs $\{(\xb^{(t)}, i_t)\}_{t=0}^\infty$ with random variable $i_t \in [n]$, which satisfies
\begin{align}
    \xb^{(t+1)} \in \text{Lin}\{\xb^{(0)},\dots,\xb^{(t)}, \nabla f_{i_0}(\xb^{(0)}),\dots,\nabla f_{i_t}(\xb^{(t})\}.\notag
\end{align}
\end{definition}

It can be easily checked that most first-order primal finite-sum optimization algorithms, such as SAG \citep{roux2012stochastic}, SVRG \citep{johnson2013accelerating}, SAGA \citep{defazio2014saga} and Katyusha \citep{allen2017katyusha}, KatyushaX \citep{allen2018katyushax}, are linear-span randomized first-order algorithms.

In this work, we prove the lower bounds by constructing adversarial functions which are ``hard enough" for any linear-span randomized first-order algorithms. To demonstrate the construction of adversarial functions, we first introduce the following quadratic function class, which comes from \citet{nesterov2013introductory}.
\begin{definition}\label{def: Q}
Let $Q(\xb; \xi, m, \zeta): \RR^m \rightarrow \RR$ be:
\begin{align}
    Q(\xb; \xi,m,\zeta) 
    &:= \frac{\xi}{2}(\xb_1 - 1)^2 + \frac{1}{2}\sum_{t=1}^{m-1}(\xb_{t+1} - \xb_t)^2 + \frac{\zeta}{2}(\xb_m)^2.\notag
\end{align}
\end{definition}
In our construction, we need the following two important properties of $Q(\xb;\xi,m,\zeta)$. 
\begin{proposition}\label{index-move}
For any $0 \leq \xi, \zeta \leq 1$ and $m \geq 1$,  the following properties hold:
\begin{enumerate}
    \item $Q(\xb; \xi, m, \zeta) \in \cS^{(0,4)}$.
    \item Suppose that $\Ub \in \RR^{m \times d}$ satisfying $\Ub\Ub^\top = \Ib$. Suppose that $\Ub = [\ub^{(1)},...\ub^{(m)}]^\top$. Then for any $\bar\xb$ satisfying $\Ub\bar\xb \in \text{Lin}\{\ub^{(1)},...,\ub^{(t)}\}$, and any differentiable function $\mu:\RR\rightarrow \RR$, we have $\nabla[Q(\Ub\bar \xb;\xi,m,\zeta) + \sum_{i=1}^m \mu(\bar \xb^\top\ub^{(i)})] \in \text{Lin}\{\ub^{(1)},...,\ub^{(t+1)}\}$.
\end{enumerate}
\end{proposition}
In short, the first property of $Q(\xb; \xi, m, \zeta)$ says that $Q$ is a convex function with $4$-smoothness, and the second property says that for any orthogonal matrix $\Ub$, the composite function $Q(\Ub\xb;\xi,m,\zeta) + \sum_{i=1}^m \mu(\xb^\top\ub^{(i)})$ enjoys the so-called \emph{zero-chain} property \citep{carmon2017lower}: if the current point is $\bar\xb$, then the information brought by an IFO call at the current point can at most increase the dimension of lienar space which $\bar\xb$ belongs to by $1$, which is very important for the proof of lower bounds. 

Based on Definition \ref{def: Q}, one can define the following three function classes: $\fsc$, $\fc$ from \citet{nesterov2013introductory} and $\fnc$ from \citet{carmon2017lower2}. 
We first introduce a class of strongly convex functions $\fsc$, which is originally defined in \citet{nesterov2013introductory}.
\begin{definition}\citep{nesterov2013introductory}
Let $\fsc(\xb; \alpha, m): \RR^m \rightarrow \RR$ be
\begin{align}
    &\fsc(\xb; \alpha, m) := \frac{1-\alpha}{4}Q\bigg(\xb; 1,m,\frac{2\sqrt{\alpha}}{\sqrt{\alpha}+1}\bigg) + \frac{\alpha}{2}\|\xb\|_2^2.\label{sc_def}
\end{align}
\end{definition}
For $\fsc(\xb; \alpha, m)$, we have the following properties.
\begin{proposition}[Chapter 2.1.4, \citet{nesterov2013introductory}]\label{property_sc}
For any $0 \leq \alpha \leq 1$, let $q:=(1-\sqrt{\alpha})/(1+\sqrt{\alpha})$, it holds that
\begin{enumerate}
    \item $\fsc(\xb; \alpha, m) \in \cS^{(\alpha,1)}$. 
    \item $\fsc(0; \alpha,m) - \inf_{\xb \in \RR^m}\fsc(\xb;\alpha,m) \leq q^2(1-q^2)$.
    \item For any $\xb$ satisfying $\xb_m = 0$, we have
    \begin{align}
        \fsc(\xb;\alpha,m) - \inf_\xb\fsc(\xb;\alpha,m) \geq \frac{\alpha}{2}q^{2m+2}.\notag
    \end{align}
\end{enumerate}
\end{proposition}
Next we introduce a class of general convex functions $\fc(\xb; m)$, which is also defined in \citet{nesterov2013introductory}.
\begin{definition}\citep{nesterov2013introductory}
Let $\fc(\xb; m):\RR^{2m-1}\rightarrow 1$ be
\begin{align}
    \fc (\xb; m) 
    :& = \frac{1}{4}Q(\xb; 1,2m-1,1).\label{fc_def}
\end{align}
\end{definition}
We have the following properties about $\fc(\xb; m)$.
\begin{proposition}[Chapter 2.1.2, \citet{nesterov2013introductory}]\label{property_c}
We have 
\begin{enumerate}
    \item $f_{\cN\text{c}} (\xb; m) \in \cS^{(0,1)}$.
     \item Let $\cX^* = \argmin_{\xb \in \RR^d}\fc(\xb;m)$ be the optimal solution set, we have $\dist^2(0,\cX^*) \leq 2m/3$.
    \item For any $\xb$ which satisfies that $|\xb_m|=...=|\xb_{2m-1}|=0$, we have $f_{\cN\text{c}} (\xb; m) - \inf_{\xb\in \RR^{2m-1}} f_{\cN\text{c}} (\xb; m) \geq 1/(16m)$.
\end{enumerate}
\end{proposition}
The above two function classes $\fsc$ and $\fc$ will be used to prove the lower bounds for convex optimization. Finally we introduce $\fnc$, which is original proposed in \citet{carmon2017lower2}, and we will use it to prove the lower bounds for nonconvex optimization.
\begin{definition}
Let $\fnc(\xb; \alpha,m):\RR^{m+1}\rightarrow \RR$ be 
\begin{align}
    \fnc (\xb; \alpha, m) :
    & = Q(\xb; \sqrt{\alpha}, m+1,0) + \alpha\Gamma(\xb),\notag
\end{align}
where $\Gamma(\xb): \RR^{m+1}\rightarrow \RR$ is defined as
\begin{align}
    \Gamma(\xb) := \sum_{i=1}^{m}120\int_1^{\xb_i}\frac{t^2(t-1)}{1+t^2}dt.\notag
\end{align}
\end{definition}
We have the following properties about $\fnc$.
\begin{proposition}[Lemmas 2, 3, 4, \citet{carmon2017lower2}]\label{property_nc}
For any $0 \leq \alpha \leq 1$, it holds that
\begin{enumerate}
    \item $\Gamma(\xb) \in \cS^{(-\lg,\lg)}$ and $\fnc(\xb;\alpha,m) \in \cS^{(-\alpha\lg,4+ \alpha\lg)}$.
    \item $\fnc(0; \alpha,m) - \inf_{\xb\in \RR^{m+1}} \fnc(\xb;\alpha,m) \leq \sqrt{\alpha}/2+10\alpha m$.
    \item For $\xb$ which satisfies that $\xb_m = \xb_{m+1} = 0$, we have $\|\nabla \fnc(\xb; \alpha,m)\|_2 \geq \alpha^{3/4}/4$.
\end{enumerate}
\end{proposition}
\section{Main Results}\label{sec:mainresult}
In this section we present our lower bound results. We start with the sum-of-nonconvex (but convex) optimization setting, then move on to the general nonconvex finite-sum optimization setting.

\subsection{$F$ is Convex -- Suboptimal Solution}
We first show the lower bounds for $F \in \cS^{(l,L)}$ with $l \geq 0$, and our goal is to find an $\epsilon$-suboptimal solution. We first show the result when $F$ is $\sigma$-strongly convex and $\{f_i\}_{i=1}^n \in \cV^{(L)}$. 
\begin{theorem}\label{linear_theorem_exp_sc}
For any linear-span randomized first-order algorithm $\cA$ and any $L,\sigma, n, \Delta, \epsilon$ such that $\epsilon \leq 8\Delta n^{7/4}\sigma^{3/2}L^{-3/2}$, there exist a dimension $d = O(n+n^{3/4}\sqrt{L/\sigma}\log(1/\epsilon))$ and functions $\{f_i\}_{i=1}^n:\RR^d \rightarrow \RR$ which satisfy that $\{f_i\}_{i=1}^n \in \cV^{(L)}$, $F \in \cS^{(\sigma,L)}$ and $F(\xb^{(0)}) - \inf_{\xb \in \RR^d} F(\xb) \leq \Delta$. In order to find $\hat\xb \in \RR^d$ such that $\EE F(\hat\xb) - \inf_{\xb \in\RR^d}F(\xb) \leq \epsilon$, $\cA$ needs at least
\begin{align}
    \Omega\bigg(n + n^{3/4}\sqrt{\frac{L}{\sigma}}\log\bigg(\frac{n\sigma\Delta}{L\epsilon}\bigg)\bigg)\label{mainth_1}
\end{align}
IFO calls.
\end{theorem}
Next we show the result when $F$ is convex and $\{f_i\}_{i=1}^n \in \cV^{(L)}$. 
\begin{theorem}\label{linear_theorem_exp_c}
For any linear-span randomized first-order algorithm $\cA$ and any $L,n, B, \epsilon$ such that $\epsilon \leq LB^2/4$ there exist a dimension $d = O(n+n^{3/4}\sqrt{L/\epsilon})$ and functions $\{f_i\}_{i=1}^n:\RR^d \rightarrow \RR$ which satisfy that $\{f_i\}_{i=1}^n \in \cV^{(L)}$, $F \in \cS^{(0,L)}$, and $\dist(\xb^{(0)}, \cX^*) \leq B$ where $\cX^* = \argmin_{\xb \in \RR^d}F(\xb)$. In order to find $\hat\xb \in \RR^d$ such that $\EE F(\hat\xb) - \inf_{\xb \in\RR^d}F(\xb) \leq \epsilon$, $\cA$ needs at least 
\begin{align}
    \Omega \bigg(n+ n^{3/4}B\sqrt{\frac{L}{\epsilon}}\bigg)\label{mainth_2}
\end{align}
IFO calls.
\end{theorem}
\begin{remark}
Our lower bounds \eqref{mainth_1} and \eqref{mainth_2} are tight, because they have been achieved by {SDCA without Duality} \citep{shalev2016sdca} for $l = \sigma$ and {KatyushaX} \citep{allen2018katyushax} for $ l =\sigma$ and $l = 0$ up to a logarithm factor. 
\end{remark}
\begin{remark}
It is interesting to compare \eqref{mainth_1} and \eqref{mainth_2} with the corresponding lower bounds for convex finite-sum optimization in \citet{woodworth2016tight}, which proves $\tilde\Omega(n+\sqrt{nL/\sigma})$ lower bound for strongly convex functions and $\Omega(n+\sqrt{nL/\epsilon})$ for convex functions. The dependence on $n$ is $n^{3/4}$ in the nonconvex case when $\epsilon \ll 1$, as opposed to $n^{1/2}$ in the (strongly) convex case. This suggests a fundamental gap. This gap has been observed firstly by \citet{shalev2016sdca} from view of the upper bounds. Our lower bound results suggest that such a gap cannot be removed.  
\end{remark}

\subsection{$F$ is Nonconvex -- Approximate Stationary Point}
Next we show the lower bounds when $F$ is $\sigma$-almost convex. For this case our goal is to find an $\epsilon$-approximate stationary point. We first present the lower result when $\{f_i\}_{i=1}^n \in \cV^{(L)}$ .
\begin{theorem}\label{linear_theorem_exp_nc}
For any linear-span randomized first-order algorithm $\cA$ and any $L,n, \Delta, \epsilon$ with $\epsilon^2 \leq (\Delta\sigma\land L\Delta n^{-1/2})/10^5$, there exist a dimension $d = O(\Delta/\epsilon^2\cdot(n^{3/4}\sqrt{\sigma L}\land \sqrt{n}L))$ and functions $\{f_i\}_{i=1}^n:\RR^d \rightarrow \RR$ which satisfy that $\{f_i\}_{i=1}^n \in \cV^{(L)}$, $F \in \cS^{(-\sigma, L)}$ and $F(\xb^{(0)}) - \inf_{\xb \in \RR^d} F(\xb) \leq \Delta$. In order to find $\hat\xb \in \RR^d$ such that $\EE\|\nabla F(\hat\xb)\|_2 \leq \epsilon$, $\cA$ needs at least
\begin{align}
    \Omega \bigg(\frac{\Delta}{\epsilon^2}\big[n^{3/4}\sqrt{\sigma L}\land \sqrt{n}L\big]\bigg)\label{mainth_3}
\end{align}
IFO calls.
\end{theorem}
\begin{remark}
Our lower bound \eqref{mainth_3} is tight for the following reasons. \eqref{mainth_3} becomes $\Omega (\Delta/\epsilon^2 \cdot n^{3/4}\sqrt{\sigma L})$ when $\sigma =O(L/\sqrt{n})$, and such IFO complexity has been achieved by {RepeatSVRG} up to a logarithm factor \citep{carmon2018accelerated, agarwal2017finding}. For the case $\sigma = \Omega(L/\sqrt{n})$, \eqref{mainth_3} becomes $\Omega(\Delta/\epsilon^2\cdot \sqrt{n}L)$, and such IFO complexity has been achieved by {SPIDER} \citep{fang2018spider} and {SNVRG} \citep{zhou2018stochastic} up to a logarithm factor.
\end{remark}

Next we show lower bounds under a slightly stronger assumption that each $f_i \in \cS^{(-\sigma, L)}$. Our result shows that with such a stronger assumption, the optimal dependency on $n$ will be smaller.
\begin{theorem}\label{linear_theorem_sep_nc}
For any linear-span randomized first-order algorithm $\cA$ and any $L,n, \Delta, \epsilon$ which satisfies that $\epsilon^2 \leq (\Delta Ln^{-1} \land \Delta\sigma)/10^3$, there exist a dimension $d = O(\Delta/\epsilon^2\cdot(\sqrt{n\sigma L} \land L))$ and functions $\{f_i\}_{i=1}^n:\RR^d \rightarrow \RR$ which satisfy that each $f_i \in \cS^{(-\sigma, L)}$ and $F(\xb^{(0)}) - \inf_{\xb \in \RR^d} F(\xb) \leq \Delta$. In order to find $\hat\xb \in \RR^d$ such that $\EE\|\nabla F(\hat\xb)\|_2 \leq \epsilon$, $\cA$ needs at least
\begin{align}
    \Omega\bigg(\frac{\Delta}{\epsilon^2}\big[\sqrt{n\sigma L}\land L\big]\bigg)\label{mainth_4}
\end{align}
IFO calls.
\end{theorem}
\begin{remark}
Our lower bound \eqref{mainth_4} is tight for the case $\sigma = O(L/n)$, where \eqref{mainth_4} becomes $\Omega(\Delta/\epsilon^2\cdot \sqrt{n\sigma l})$. Such IOF complexity has been achieved by {Natasha} \citep{allen2017natasha}, {RapGrad} \citep{lan2018accelerated} and {StagewiseKatyusha} \citep{yang2018does} up to a logarithm factor. Nevertheless, for the case $\sigma = \Omega(L/n)$, \eqref{mainth_4} becomes $\Omega(\Delta/\epsilon^2\cdot L)$, which does not match the best-known upper bound $O(\Delta/\epsilon^2 \cdot \sqrt{n}L)$ \citep{fang2018spider} by a factor of $\sqrt{n}$ on the dependency of $n$. We leave it as a future work to close this gap. 
\end{remark}

\subsection{Discussion on the Average Smoothness Assumption}

Careful readers may have already found that in our Theorems \ref{linear_theorem_exp_sc}, \ref{linear_theorem_exp_c} and \ref{linear_theorem_exp_nc}, we only assume that $\{f_i\}_{i=1}^n \in \cV^{(L)}$. In other words, the above lower bound results (except Theorem \ref{linear_theorem_sep_nc}) hold for $\{f_i\}_{i=1}^n$ that is average smooth. Nevertheless, most of the upper bound results achieved by existing finite-sum optimization algorithms (i.e., SDCA without Duality \citep{shalev2016sdca}, Natasha \citep{allen2017natasha}, KatyushaX \citep{allen2018katyushax}, RapGrad \citep{lan2018accelerated}, StagewiseKatyusha \citep{yang2018does} and RepeatSVRG \citep{agarwal2017finding, carmon2018accelerated}) are proved under the assumption that $f_i \in \cS^{(-L,L)}$ for each $i \in [n]$, which is stronger than assuming $\{f_i\}_{i=1}^n \in \cV^{(L)}$, which only appears in \citet{zhou2018stochastic} and \citet{fang2018spider}.
Therefore, it is important to verify that these upper bounds results still hold under the weaker assumption that $\{f_i\}_{i=1}^n$ that is average smooth.

To verify this, we need to rethink about the role that the assumption $f_i \in \cS^{(-L,L)}$ for each $i \in [n]$ plays in the convergence analyses for those algorithms. In detail, in the convergence analyses of  those nonconvex finite-sum optimization algorithms including {SDCA without Duality} \citep{shalev2016sdca}, {Natasha} \citep{allen2017natasha}, {KatyushaX} \citep{allen2018katyushax}, one needs the assumption that $f_i \in \cS^{(-L,L)}$ for each $i \in [n]$ in the following two scenarios:
First, it is used to show that $F \in \cS^{(-L,L)}$, which can be derived as follows: for any $\xb, \yb \in \RR^d$,
\begin{align}
    \|\nabla F(\xb) - \nabla F(\yb)\|_2^2 &\leq \EE_i\|\nabla f_i(\xb) - \nabla f_i(\yb)\|_2^2 \notag \\
    &\leq L^2\|\xb - \yb\|_2^2.\label{dis_1}
\end{align}
Second, it is used to upper bound the variance of the semi-stochastic gradient at each iteration, which is an unbiased estimator of the true gradient. More specifically, let $\vb$ be
\begin{align}
    \vb = \nabla f_i(\xb) - \nabla f_i(\hat\xb) + \nabla F(\hat\xb),\notag
\end{align}
where $\hat\xb$ is the global minimum of $F$ when $F$ is convex or any snapshot of $\xb$ when $F$ is nonconvex. Then we have
\begin{align}
    \EE_i\|\vb - \nabla F(\xb)\|_2^2
    &= 
    \EE_i\|\nabla f_i(\xb) - \nabla f_i(\hat\xb) + \nabla F(\hat\xb) - \nabla F(\xb)\|_2^2\notag \\
    & \leq 
    2 \Big[\EE_i\|\nabla f_i(\xb) - \nabla f_i(\hat\xb)\|_2^2 + \|\nabla F(\hat\xb) - \nabla F(\xb)\|_2^2\Big]\notag  \\
    &\leq 2 L^2 \|\xb - \hat\xb\|_2^2.\label{dis_0}
\end{align}

We can see that in both scenarios, the weaker assumption $\{f_i\}_{i=1}^n \in \cV^{(L)}$ is sufficient to make \eqref{dis_1} and \eqref{dis_0} hold. Thus, we make the following informal statement, which may be regarded as a slight improvement/modification in terms of assumptions over existing algorithms for nonconvex finite-sum optimization problems.
\begin{proposition}\label{prop:notaffect}
For existing nonconvex finite-sum optimization algorithms including SDCA without Duality \citep{shalev2016sdca}, Natasha \citep{allen2017natasha}, KatyushaX \citep{allen2018katyushax}, RapGrad \citep{lan2018accelerated}, StagewiseKatyusha \citep{yang2018does} and RepeatSVRG \citep{agarwal2017finding, carmon2018accelerated}, we can replace the smoothness assumption that $f_i \in \cS^{(-L,L)}$ with $\{f_i\}_{i=1}^n \in \cV^{(L)}$, without affecting their IFO complexities.
\end{proposition}

\section{Proof of Main Theorems}\label{sec:main}

In this section, we provide the detailed proofs for the lower bounds presented in Section \ref{sec:mainresult}.
Due to space limit, we only provide the proofs for Theorems \ref{linear_theorem_exp_sc} and \ref{linear_theorem_exp_nc}, and defer the proofs for the other theorems in the supplementary material.

\subsection{Technical Lemmas}
Our proofs are based on the following three technical lemmas, whose proofs can be found in the supplementary material.

The first lemma provides the upper bound for the average smoothness parameter of finite-sum functions, when each component function is lower and upper smooth.
\begin{lemma}\label{avsmooth}
For any $g:\RR^{m}\rightarrow \RR$ and $g \in \cS^{(\xi, \zeta)}$ where $0 \leq |\xi| \leq \zeta$, suppose that $\{\Ub^{(i)}\}_{i=1}^n \in \cO(m,mn, n)$. Then for $\bar g_i:\RR^{mn}\rightarrow \RR$ where $\bar g_i(\xb): = \sqrt{n} g(\Ub^{(i)}\xb)$, we have that $\{\bar g_i\}_{i=1}^n \in \cV^{(\zeta)}$. For $\bar G(\xb) = \sum_{i=1}^n\bar g_i(\Ub^{(i)}\xb)/n$, we also have $\bar G \in \cS^{(\xi/\sqrt{n}, \zeta)}$.
\end{lemma}
In the proof we need to do scale transformation to the given functions. The following lemma describes how problem dependent quantities change with respect to scale transformation.
\begin{lemma}\label{scale_change}
Let $\{\bar g_i\}_{i=1}^n, \bar g_i: \RR^d \rightarrow \RR$ be functions satisfying $\{\bar g_i\}_{i=1}^n \in \cV^{(L')}$, $\bar g_i \in \cS^{(\xi', \zeta')}$. We further define $\bar G = \sum_{i=1}^n \bar g_i/n$, and $\cZ^* = \argmin_{\zb \in \RR^d} \bar G(\zb)$. Suppose that $\bar G(0) - \inf_{\xb\in \RR^d}\bar G(\xb) \leq \Delta'$ and $\dist(0, \cZ^*) \leq B'$. For any $\lambda, \beta>0$, we define $\{g_i\}_{i=1}^n$ satisfying $g_i(\xb) = \lambda \bar g_i(\xb/\beta)$ and $G = \sum_{i=1}^n g_i/n$. Let $(\cZ')^* = \argmin_{\zb \in \RR^d} G(\zb)$. Then we have that $\{g_i\}_{i=1}^n \in \cV^{(\lambda/\beta^2\cdot L')}$, $g_i \in \cS^{(\lambda/\beta^2\cdot \xi', \lambda/\beta^2\cdot \zeta')}$, $G(0) - \inf_{\xb \in \RR^d}G(\xb) \leq \lambda\Delta'$ and $\dist(0, (\cZ')^*) \leq \beta B'$.
\end{lemma}
We also need the following lemmas to guarantee an $\Omega(n)$ lower bound for finding an $\epsilon$-suboptimal solution when $F$ is strongly convex or convex.
\begin{lemma}\label{omega_n}
For any linear-span randomized first-order algorithm $\cA$ and any $L, \sigma, n, \Delta, \epsilon$ with $\epsilon <\Delta/4 $, there exist functions $\{f_i\}_{i=1}^n: \RR^n \rightarrow \RR$ and $F = \sum_{i=1}^n f_i/n$ which satisfy that $\{f_i\}_{i=1}^n \in \cV^{(L)}$, $F \in \cS^{(\sigma, L)}$ and $F(\xb^{(0)}) - \inf_{\xb \in \RR^n} F(\xb) \leq \Delta$. In order to find $\hat\xb \in \RR^n$ such that $\EE F(\hat\xb) - \inf_{\xb \in \RR^n}F(\xb) \leq \epsilon$, $\cA$ needs at least $\Omega(n)$ IFO calls.
\end{lemma}

\begin{lemma}\label{omega_n_convex}
For any linear-span randomized first-order algorithm $\cA$ and any $L, \sigma, n, \Delta, \epsilon$ with $\epsilon <\Delta/4 $, there exist functions $\{f_i\}_{i=1}^n: \RR^n \rightarrow \RR$ and $F = \sum_{i=1}^n f_i/n$ which satisfy that $\{f_i\}_{i=1}^n \in \cV^{(L)}$, $F \in \cS^{(0, L)}$ and $F(\xb^{(0)}) - \inf_{\xb \in \RR^n} F(\xb) \leq \Delta$. In order to find $\hat\xb \in \RR^n$ such that $\EE F(\hat\xb) - \inf_{\xb \in \RR^n}F(\xb) \leq \epsilon$, $\cA$ needs at least $\Omega(n)$ IFO calls.
\end{lemma} 
We now begin our proof. Without loss of generality, we assume that $\xb^{(0)} = \zero$, otherwise we can replace function $f(\xb)$ with $\hat f(\xb) = f(\xb - \xb^{(0)})$.
\subsection{Proofs for: $F$ is Convex}

\begin{proof}[Proof of Theorem \ref{linear_theorem_exp_sc}]
Let $\{\Ub^{(i)}\}_{i=1}^n \in \cO(T,Tn, n)$. We choose $\bar f_i(\xb): \RR^{Tn}\rightarrow \RR$ as follows:
\begin{align}
    \bar f_i(\xb) &:= \sqrt{n}\fsc(\Ub^{(i)}\xb; \alpha,T), \notag \\
    \bar F(\xb) &:= \frac{1}{n}\sum_{i=1}^n \bar f_i(\xb).\notag
\end{align}
First, we claim that $\{\bar f_i(\xb)\}_{i=1}^n \in \cV^{(1)}$ and $ \bar F \in \cS^{(\alpha/\sqrt{n}, 1)}$ due to Lemma \ref{avsmooth} where $\fsc \in \cS^{(\alpha,1)}$, $\alpha \leq 1$. 
Next, we claim $\bar F(0) - \inf_\xb \bar F(\xb) \leq 1/\sqrt{n}\sum_{i=1}^n [\fsc(0; \alpha,T) - \inf_\xb\fsc(\Ub^{(i)}\xb; \alpha,T)] \leq q^2/\sqrt{n}(1-q^2)$, 
because
\begin{align}
    \bar F(0) - \inf_\xb \bar F(\xb) 
    &= \frac{\sqrt{n}}{n}\sum_{i=1}^n\fsc(0;\alpha,T) - \inf_\xb\frac{\sqrt{n}}{n}\sum_{i=1}^n\fsc(\Ub^{(i)}\xb;\alpha,T)\notag \\
    & = \frac{1}{\sqrt{n}}\sum_{i=1}^n [\fsc(0;\alpha,T) - \inf_\xb \fsc(\xb;\alpha,T)]\notag \\
    & \leq \frac{q^2}{\sqrt{n}(1-q^2)},\notag
\end{align}
where the second equality holds due to the fact that $\inf_\xb\sum_{i=1}^n \fsc(\Ub^{(i)}\xb;\alpha,T) = \sum_{i=1}^n \inf_\xb \fsc(\xb;\alpha,T)$.
Finally, 
let $\yb^{(i)} = \Ub^{(i)}\xb$. If there exists $\cI \subset [n], |\cI| > n/2$ and for each $i \in \cI$, $\yb^{(i)}_T = 0$. Then, by Proposition \ref{property_sc}, for each $i \in \cI$, we have $\fsc(\yb^{(i)}, \alpha, T) - \inf_\zb\fsc(\zb, \alpha, T) \geq \alpha q^{2T+2}/2$, which implies
\begin{align}
     \bar F(\xb) - \inf_\zb \bar F(\zb) 
     &\geq \frac{1}{\sqrt{n}}\sum_{i\in \cI} [\fsc(\yb^{(i)}, \alpha, T) - \inf_\zb\fsc(\zb, \alpha, T)]\notag \\
     &\geq \alpha\sqrt{n}q^{2T+2}/2.\label{temp_0}
\end{align}
With the above properties, we can choose $f_i(\xb) = \lambda\bar f_i(\xb/\beta)$ in the following proof. We first consider any fixed index sequence $\{i_t\}$. In the sequel, we consider two cases: (1) $\sqrt{n}\sigma/L \leq 1/4$; and (2) $\sqrt{n}\sigma/L > 1/4$.  

\noindent\textbf{Case (1):}  $\sqrt{n}\sigma/L \leq 1/4$, we set $\alpha, \lambda, \beta, T$ as follows
\begin{align}
    \alpha &= \frac{\sqrt{n}\sigma}{L}\notag\\
    \lambda &=\frac{4\sqrt{n\alpha}\Delta}{(1-\sqrt{\alpha})^2}\notag\\
    \beta &= \sqrt{\lambda/L}\notag\\
    T& = \sqrt{\frac{L}{\sqrt{n}\sigma}}\cdot \log\bigg[ \bigg(\frac{\sigma}{L}\bigg)^{3/2}\frac{8n^{7/4}\Delta}{\epsilon}\bigg].\notag 
\end{align}
Then by Lemma \ref{scale_change}, we have that $\{f_i\}_{i=1}^n \in \cV^{(L)}$ , $F \in \cS^{(\sigma, L)}$, $F(\zero) - \inf_\zb F(\zb) \leq \Delta$ due to $\sqrt{n}\sigma/L \leq 1/4$. By Proposition \ref{index-move}, we know that for any algorithm output $\xb^{(t)}$ where $t$ is less than 
\begin{align}
    \frac{nT}{2} = n^{3/4}\sqrt{\frac{L}{\sigma}}\log\bigg[ \bigg(\frac{\sigma}{L}\bigg)^{3/2}\frac{8n^{7/4}\Delta}{\epsilon}\bigg],\label{temp2}
\end{align}
there exists $\cI \subset [n], |\cI| > n - nT/(2T) = n/2$ and for each $i \in \cI$, $\yb^{(i)}_T = 0$, where $\yb^{(i)} = \Ub^{(i)}\xb^{(t)}$. Thus, $\xb^{(t)}$ satisfies 
\begin{align}
    F(\xb^{(t)}) - \inf_{\zb}F(\zb) \geq \lambda\alpha\sqrt{n}q^{2T+2}/2 \geq \epsilon,\notag
\end{align}
where the first inequality holds due to \eqref{temp_0}. Then, applying Yao's minimax theorem \citep{yao1977probabilistic}, we have that for any randomized index sequence $\{i_t\}$, we have the lower bound \eqref{temp2}.
\noindent\textbf{Case (2):} $\sqrt{n}\sigma/L > 1/4$, by Lemma \ref{omega_n} we know that there exists an $\Omega(n)$ lower bound. 

By combining Cases (1) and (2), we have the lower bound \eqref{mainth_1}.
\end{proof}

\begin{proof}[Proof of Theorem \ref{linear_theorem_exp_c}]
Let $\{\Ub^{(i)}\}_{i=1}^n \in \cO(2T-1,(2T-1)n, n)$. We choose $\bar f_i(\xb): \RR^{Tn}\rightarrow \RR$ as follows:
\begin{align}
    \bar f_i(\xb) &:= \sqrt{n}\fc(\Ub^{(i)}\xb; \alpha,T), \notag \\
    \bar F(\xb) &:= \frac{1}{n}\sum_{i=1}^n \bar f_i(\xb).\notag
\end{align}
We have the following properties. First, we claim that $\{\bar f_i(\xb)\} \in \cV^{(1)}$ because of Lemma \ref{avsmooth} where $\fc \in \cS^{(0,1)} \subset \cS^{(-1,1)}$. 
Next, suppose that $\bar\cX^* = \argmin_\zb \bar F(\zb)$, then by definition, we have that for any $\bar\xb^* \in \bar \cX^*$, $\Ub^{(i)}\bar\xb^* \in (\cX^*)^{(i)}$, where $(\cX^*)^{(i)} = \argmin_\zb \fc(\zb; \alpha, T)$. Thus, we have 
\begin{align}
    \dist^2(0, \bar\cX^*) &= \inf_{\bar\xb^* \in \bar\cX^*}\|0 - \bar\xb^*\|_2^2 = \inf_{\bar\xb^* \in \bar\cX^*}\sum_{i=1}^n\|\Ub^{(i)}\bar\xb^*\|_2^2 \leq \frac{2nT}{3} \leq nT.\notag
\end{align}
Finally, 
let $\yb^{(i)} = \Ub^{(i)}\xb \in \RR^T$. If there exists $\cI \subset [n], |\cI| > n/2$ and for each $i \in \cI$, $\yb^{(i)}_T =... = \yb^{(i)}_{2T-1}=0$, then by Proposition \ref{property_c}, we have
\begin{align}
     &\bar F(\xb) - \inf_\zb \bar F(\zb) \notag \\
     &\geq \frac{1}{\sqrt{n}}\sum_{i\in \cI} [\fsc(\yb^{(i)}, \alpha, T) - \inf_\zb\fsc(\zb, \alpha, T)] \notag \\
     &\geq \sqrt{n}/(16T).\label{temp456}
\end{align}
With above properties, we set the final functions as $f_i(\xb) = \lambda\bar f_i(\xb/\beta)$. We first consider any fixed index sequence $\{i_t\}$. For the case $\epsilon \leq LB^2/(16\sqrt{n})$, we set $\lambda, \beta, T$ as
\begin{align}
    \lambda =\frac{B\sqrt{16\epsilon L}}{n^{3/4}},
    \beta = \sqrt{\lambda/L},
    T = \frac{B\sqrt{L}}{4n^{1/4}\epsilon^{1/2}},\notag 
\end{align}
Since Then by Lemma \ref{scale_change}, we have that $f_i\in \cV^{(L)}$, $F \in \cS^{(0,L)}$, $F(0) - \inf_\zb F(\zb) \leq \Delta$. By Proposition \ref{index-move}, we know that for any algorithm output $\xb^{(t)}$ where $t$ is less than 
\begin{align}
    \frac{nT}{2} = 8n^{3/4}B\sqrt{\frac{L}{\epsilon}}\label{temp5},
\end{align}
there exists $\cI \subset [n], |\cI| > n - nT/(2T) = n/2$ and for each $i \in \cI$, $\yb^{(i)}_T =... = \yb^{(i)}_{2T-1}=0$, where $\yb^{(i)} = \Ub^{(i)}\xb^{(t)}$. Thus, $\xb^{(t)}$ satisfies that
\begin{align}
    \bar F(\xb^{(t)}) - \inf_\zb \bar F(\zb) \geq \lambda \sqrt{n}/(16T) \geq \epsilon,\notag
\end{align}
where the first inequality holds due to \eqref{temp456}. 
Then, applying Yao's minimax theorem, we have that for any randomized index sequence $\{i_t\}$, we have the lower bound \eqref{temp5}. For the case $LB^2/4 \geq \epsilon \geq LB^2/(16\sqrt{n})$, by Lemma \ref{omega_n_convex} we know that there exists an $\Omega(n)$ lower bound. Thus, with all above statements, we have the lower bound \eqref{mainth_2}.
\end{proof}

\subsection{Proofs for: $F$ is Nonconvex}

\begin{proof}[Proof of Theorem \ref{linear_theorem_exp_nc}]
Let $\{\Ub^{(i)}\}_{i=1}^n \in \cO(T+1,(T+1)n, n)$. We choose $\bar f_i(\xb): \RR^{Tn}\rightarrow \RR$ as follows:
\begin{align}
    \bar f_i(\xb) &:= \sqrt{n}\fnc(\Ub^{(i)}\xb; \alpha,T),\notag \\
    \bar F(\xb) &:= \frac{1}{n}\sum_{i=1}^n \bar f_i(\xb).
\end{align}
 We have the following properties. First, we claim that $\{\bar f_i(\xb)\}_{i=1}^n \in \cV^{(4+\alpha\lg)}$ and $\bar F(\xb) \in \cS^{(-\alpha\lg/\sqrt{n}, 4+\alpha\lg)}$ by Lemma \ref{avsmooth} where $\fnc \in \cS^{(-\alpha\lg, 4+\alpha\lg)}$ and $\alpha\lg<4+\alpha\lg$.
Next, we have 
\begin{align}
    \bar F(0) - \inf_\xb \bar F(\xb)
    &\leq 1/\sqrt{n}\sum_{i=1}^n [\fsc(0; \alpha,T) - \inf_\xb\fsc(\Ub^{(i)}\xb; \alpha,T)] \notag \\
    &\leq \sqrt{n}(\sqrt{\alpha} + 10\alpha T).
\end{align}
Finally, 
let $\yb^{(i)} = \Ub^{(i)}\xb$. If there exists $\cI, |\cI| > n/2$ and for each $i \in \cI$, $\yb^{(i)}_T = \yb^{(i)}_{T+1} = 0$, then by Proposition \ref{property_nc}, we have 
\begin{align}
     \|\nabla \bar F(\xb)\|_2^2 &\geq \frac{1}{n}\sum_{i\in \cI} \|(\Ub^{(i)})^\top\nabla[\fnc(\Ub^{(i)}\xb;\alpha, T)]\|_2^2 \notag \\
     &\geq \frac{1}{n}\frac{n}{2}(\alpha^{3/4}/4)^2 \notag \\
     &= \alpha^{3/2}/32.\label{temp111}
\end{align}
With above properties, we choose $f_i(\xb) = \lambda\bar f_i(\xb/\beta)$ in the following proof. We first consider any fixed index sequence $\{i_t\}$. We set $\alpha, \lambda, \beta, T$ as
\begin{align*}
    \alpha &= \min\bigg\{\frac{5\sigma \sqrt{n}}{\lg L}, \frac{1}{\lg}\bigg\}\\
    \lambda &= \frac{5\epsilon^2}{L\alpha^{3/2}}\\
    \beta &= \sqrt{5\lambda/L}\\
    T&  = \frac{L\Delta}{55\sqrt{n}\epsilon^2}\sqrt{\min\bigg\{\frac{5\sigma \sqrt{n}}{\lg L}, \frac{1}{\lg}\bigg\}},
\end{align*}
Then by Lemma \ref{scale_change}, we have that $\{f_i\}_{i=1}^n \in \cV^{(L)}$, $F \in \cS^{(\sigma, L)}$, $F(0) - \inf_\zb F(\zb) \leq \Delta$ with the assumption that $\epsilon^2 \leq L\alpha\Delta/(55\sqrt{n})$. By Proposition \ref{index-move}, we know that for any algorithm output $\xb^{(t)}$ where $t$ is less than
\begin{align}
    \frac{nT}{2}  &= \frac{L\sqrt{n}\Delta}{110\epsilon^2}\sqrt{\min\bigg\{\frac{5\sigma \sqrt{n}}{\lg L}, \frac{1}{\lg}\bigg\}} ,\label{temp112}
\end{align}
there exists $\cI \subset [n], |\cI| > n - nT/(2T) = n/2$ and for each $i$, $\yb^{(i)}_T = \yb^{(i)}_{T+1} = 0$ where $\yb^{(i)} = \Ub^{(i)}\xb^{(t)}$. Thus, by \eqref{temp111}, $\xb^{(t)}$ satisfies 
\begin{align}
    \|\nabla F(\xb^{(t)})\|_2 \geq \lambda/\beta\cdot \sqrt{\alpha^{3/2}/32} \geq \epsilon.\notag
\end{align}
Then, applying Yao's minimax theorem, we have that for any randomized index sequence $\{i_t\}$, we have the lower bound \eqref{temp112}, which implies \eqref{mainth_3}.
\end{proof}

\begin{proof}[Proof of Theorem \ref{linear_theorem_sep_nc}]
Let $\{\Ub^{(i)}\}_{i=1}^n \in \cO(T+1,(T+1)n, n)$. We choose $\bar f_i(\xb): \RR^{(T+1)n}\rightarrow \RR$ as follows:
\begin{align}
    \bar f_i(\xb) :&= Q(\Ub^{(i)}\xb; \sqrt{\alpha}, T+1, 0) + \frac{\alpha}{n} \Gamma (\Ub\xb),\notag \\
    \bar F(\xb) :&= \frac{1}{n}\sum_{i=1}^n \bar f_i(\xb).\notag
\end{align}
 We have the following properties. First, we claim that each $\bar f_i \in \cS^{(-\alpha\lg/n, 4+\alpha\lg/n)}$ because $Q \in \cS^{(0,4)}$ and $\Gamma \in \cS^{(-\lg, \lg)}$. Next, note that
 \begin{align}
     \bar F(\xb) &= \frac{1}{n}\sum_{i=1}^n \bar f_i(\xb) \notag \\
     &=  \frac{1}{n}\sum_{i=1}^n[Q(\Ub^{(i)}\xb; \sqrt{\alpha}, T+1, 0) + \alpha \Gamma(\Ub^{(i)}\xb)] \notag \\
     &= \frac{1}{n}\sum_{i=1}^n\fnc(\Ub^{(i)}\xb; \sqrt{\alpha}, T+1).\notag
 \end{align}
Then we have 
\begin{align}
    &\bar F(0) - \inf_\xb \bar F(\xb) \notag \\
    &= \frac{1}{n}\sum_{i=1}^n\fnc(0;\sqrt{\alpha},T+1)  - \inf_\xb\frac{1}{n}\sum_{i=1}^n\fnc(\Ub^{(i)}\xb;\sqrt{\alpha},T+1)\notag \\
    & = \frac{1}{n}\sum_{i=1}^n [\fnc(0;\sqrt{\alpha},T+1) - \inf_\xb \fnc(\xb;\sqrt{\alpha},T+1)]\notag \\
    & \leq \sqrt{\alpha} + 10\alpha T,\notag
\end{align}
where the second equality holds due to the fact that $\inf_\xb\sum_{i=1}^n \fnc(\Ub^{(i)}\xb;\alpha,T) = \sum_{i=1}^n \inf_\xb \fnc(\xb;\alpha,T)$.
Finally, 
let $\yb^{(i)} = \Ub^{(i)}\xb$. If there exists $\cI, |\cI| > n/2$ and for each $i \in \cI$, $\yb^{(i)}_T = \yb^{(i)}_{T+1} = 0$, then by Proposition \ref{property_nc}, we have
\begin{align}
     \|\nabla \bar F(\xb)\|_2^2 &\geq \frac{1}{n^2}\sum_{i\in \cI} \|\Ub^{(i)}\nabla[\fnc(\Ub^{(i)}\xb;\alpha, T)]\|_2^2 \notag \\
     &\geq \frac{1}{n^2}\frac{n}{2}(\alpha^{3/4}/4)^2 \notag \\
     &= \alpha^{3/2}/(32n).\label{temp500}
\end{align}
With above properties, we set the final functions $f_i(\xb) = \lambda\bar f_i(\xb/\beta)$. We first consider any fixed index sequence $\{i_t\}$. We set $\alpha, \lambda, \beta, T$ as
\begin{align*}
    \alpha &= \min\bigg\{1,\frac{5n\sigma }{\lg L}\bigg\}\\
    \lambda &= \frac{160n\epsilon^2}{L\alpha^{3/2}}\\
    \beta &= \sqrt{5\lambda/L}\\
    T&  = \frac{\Delta L}{1760n\epsilon^2}\sqrt{\min\bigg\{1,\frac{5n\sigma }{\lg L}\bigg\}},
\end{align*}
Then by Lemma \ref{scale_change}, we have that $f_i \in \cS^{(-\sigma, L)}$, $F(0) - \inf_\zb F(\zb) \leq \Delta$ with the assumption that $\epsilon^2 \leq \Delta L \alpha/(1760n)$. By Proposition \ref{index-move}, we know that for any algorithm output $\xb^{t}$ where $t$ is less than
\begin{align}
    \frac{nT}{2} = \frac{\Delta L}{3520\epsilon^2}\sqrt{\min\bigg\{1,\frac{5n\sigma }{\lg L}\bigg\}},\label{501}
\end{align}
there exists $\cI\subset [n], |\cI| > n - nT/(2T) = n/2$ and for each $i$, $\yb^{(i)}_T = \yb^{(i)}_{T+1} = 0$ where $\yb^{(i)} = \Ub^{(i)}\xb^{(t)}$. Thus, by \eqref{temp500}, $\xb^{(t)}$ satisfies that
\begin{align}
    \|\nabla F(\xb^{(t)})\|_2 \geq \lambda/\beta\cdot \sqrt{\alpha^{3/2}/(32n)} \geq \epsilon. \notag
\end{align}
Applying Yao's minimax theorem, we have that for any randomized index sequence $\{i_t\}$, we have the lower bound \eqref{501}, which implies \eqref{mainth_4}.

\end{proof}

\section{Conclusions and Future Work}
In this paper we proved the lower bounds of IFO complexity for linear-span randomized first-order algorithms to find $\epsilon$-suboptimal points or $\epsilon$-approximate stationary points for smooth nonconvex finite-sum optimization, where the objective function is the average of $n$ nonconvex functions. 
We would like to consider more general setting, such as $F$ is of $(\sigma, L)$-smoothness while each $f_i$ is $(l, L)$-smoothness. We are also interested in proving lower bound results for high-order finite-sum optimization problems \citep{arjevani2017oracle, agarwal2017lower}.

\appendix

\section{Proofs of Technical Lemmas}
\subsection{Proof of Lemma \ref{avsmooth}}
\begin{proof}[Proof of Lemma \ref{avsmooth}]
For any $\xb, \yb \in \RR^{mn}$, we have that
\begin{align}
     \EE_i\|\nabla \bar g_i (\xb) - \nabla \bar g_i (\yb)\|_2^2 
    &= \frac{1}{n}\sum_{i=1}^n\|\nabla[   \sqrt{n}g(\Ub^{(i)}\xb)]  - \nabla [\sqrt{n}g(\Ub^{(i)}\yb)\|_2^2]\notag \\
    &= \sum_{i=1}^n\|[\Ub^{(i)}]^\top\nabla  g(\Ub^{(i)}\xb)- [\Ub^{(i)}]^\top\nabla g(\Ub^{(i)}\yb)\|_2^2\notag \\
    & = \sum_{i=1}^n\|\nabla  g(\Ub^{(i)}\xb)- \nabla g(\Ub^{(i)}\yb)\|_2^2\notag \\
    & \leq \beta^2\sum_{i=1}^n \|\Ub^{(i)}\xb - \Ub^{(i)}\yb\|_2^2\notag \\
    & = \beta^2\|\xb - \yb\|_2^2,\notag
\end{align}
where the third and last equality holds due to the fact that $\Ub^{(i)}[\Ub^{(i)}]^\top = \Ib$ and $\Ub^{(i)}[\Ub^{(j)}]^\top = \bm{0}$ for each $i \neq j$, and the inequality holds due to the fact that $g \in \cS^{(-\zeta, \zeta)}$. Thus, we have $\{\bar g_i\}_{i=1}^n \in \cV^{(\zeta)}$. To prove $\bar G \in \cS^{(\xi/\sqrt{n}, \zeta)}$, we have
\begin{align}
    \nabla^2 \bar G(\xb) = \frac{1}{\sqrt{n}}\sum_{i=1}^n \Ub^{(i)}(\Ub^{(i)})^\top\nabla^2 g(\Ub^{(i)}\xb) \succeq \frac{\xi}{\sqrt{n}}\Ib,\notag
\end{align}
where the inequality holds due to the assumption that $g \in \cS^{(\xi, \beta)}$. With this fact and $\|\nabla \bar G(\xb) - \nabla \bar G(\yb)\|_2^2 \leq \EE_i\|\nabla \bar g_i(\xb) - \nabla \bar g_i(\yb)\|_2^2 \leq \beta^2 \|\xb - \yb\|_2^2$ which implies that $\nabla^2 \bar G(\xb) \preceq \beta\Ib$, we conclude that $\bar G \in \cS^{(\xi/\sqrt{n}, \beta)}$.
\end{proof}

\subsection{Proof of Lemma \ref{scale_change}}
\begin{proof}[Proof of Lemma \ref{scale_change}]
First we have $\{g_i\}_{i=1}^n \in \cV^{(\lambda/\beta^2\cdot L')}$ because for any $\xb, \yb \in \RR^d$,
\begin{align}
    \EE_i\|\nabla g_i(\xb) - \nabla g_i(\yb)\|_2^2 &= \lambda^2\EE_i\|\nabla \bar g_i(\xb/\beta)/\beta - \nabla \bar g_i(\yb/\beta)/\beta\|_2^2\notag \\
    & \leq \lambda^2/\beta^2 (L')^2 \|\xb/\beta -\yb/\beta\|_2^2\notag \\
    & = (\lambda/\beta^2\cdot L')^2 \|\xb - \yb\|_2^2.\notag
\end{align}
Next we have $g_i \in \cS^{(\lambda/\beta^2\cdot \xi', \lambda/\beta^2\cdot \zeta')}$ because $\nabla^2 g_i(\xb) = \lambda/\beta^2\nabla^2 \bar g_i(\xb /\beta)$ and for any $\xb \in \RR^d$,
\begin{align}
    \lambda/\beta^2\cdot \xi'\Ib \preceq \lambda/\beta^2\nabla^2 \bar g_i(\xb /\beta) \preceq \lambda/\beta^2\cdot \zeta'\Ib.\notag 
\end{align}
Next we have $G(0) - \inf_{\xb }G(\xb) \leq \lambda\Delta'$ because 
\begin{align}
    G(0) - \inf_{\xb} G(\xb) = \lambda \bar G(0) - \lambda \inf_{\xb} G(\xb) \leq \lambda \Delta'.\notag
\end{align}
Finally we have $\dist(0, (\Zb')^*) \leq \beta B'$ because $(\Zb')^* = \beta\cdot \Zb^*$. 
\end{proof}

\subsection{Proof of Lemma \ref{omega_n}}
\begin{proof}[Proof of Lemma \ref{omega_n}]
Suppose the initial point $\xb^{(0)} = \zero$. Consider the following function $\{\bar f_i\}_{i=1}^n$, $\bar f_i:\RR^n \rightarrow \RR$, where
\begin{align}
    \bar f_i(\xb) &:= -\sqrt{n}\la\xb, \eb^{(i)}\ra + \frac{\|\xb\|_2^2}{2}, \notag \\
    \bar F(\xb)&: = \frac{1}{n}\sum_{i=1}^n \bar f_i(\xb),\notag
\end{align}
$\eb^{(i)}$ is the $i$-th coordinate vector. We have that $\{\bar f_i\}_{i=1}^n\in \cV^{(1)}$ and the global minimizer of $\bar F$ is 
\begin{align}
    \xb^* = \frac{1}{\sqrt{n}}\sum_{i=1}^n \eb^{(i)}.\notag
\end{align}
Thus we have $\dist(\zero, \xb^*) = 1$ and $\bar F(0) - \inf_{\xb} \bar F(\xb) = 1/2$. Moreover, if point $\xb$ satisfies that $|\text{supp}\{\xb\}| \leq n/2$, then
\begin{align}
    \bar F(\xb) &= \frac{\|\xb\|_2^2}{2} - \frac{1}{\sqrt{n}}\sum_{i=1}^n \la\xb,\eb^{(i)}\ra=\frac{\|\xb\|_2^2}{2} - \frac{1}{\sqrt{n}}\sum_{i \in \text{supp}\{\xb\}} \la\xb,\eb^{(i)}\ra \geq -1/4,\notag
\end{align}
which implies 
\begin{align}
    \bar F(\xb) - \inf_{\xb} F(\xb) \geq 1/4.\label{omega_n_1}
\end{align}
Next we choose $f_i = \lambda \bar f_i(\xb/\beta)$, where $\lambda = 2\Delta, \beta =\sqrt{2\Delta/L} $, then we can check that $\{f_i\}_{i=1}^n \in \cV^{(L)}$, $F(0)- \inf_{\xb}F(\xb) \leq \Delta$, $F \in \cS^{(L,L)} \subset \cS^{(\sigma,L)}$. Moreover, since $\nabla f_i(\xb) = -\lambda\sqrt{n}\eb^{(i)}/\beta + \lambda\xb/\beta^2$, then for some $\xb$, $i$ is in the support set of $\xb$ only if $f_i$ has been called. Thus, if less than $n/2$ IFO calls have been made, then current point $\xb$ satisfies that $|\text{supp}\{\xb\}| \leq n/2$. With \eqref{omega_n_1}, we have that
$F(\xb) - \inf_{\zb }F(\zb) \geq \Delta/4 \geq \epsilon $. 
\end{proof}
\subsection{Proof of Lemma \ref{omega_n_convex}}
\begin{proof}[Proof of Lemma \ref{omega_n_convex}]
Suppose the initial point $\xb^{(0)} = \zero$. Consider the following function $\{\bar f_i\}_{i=1}^n$, $\bar f_i:\RR^n \rightarrow \RR$, where
\begin{align}
    \bar f_i(\xb) &:= -\sqrt{n}\la\xb, \eb^{(i)}\ra + \frac{\|\xb\|_2^2}{2}, \notag \\
    \bar F(\xb) &:= \frac{1}{n}\sum_{i=1}^n \bar f_i(\xb), \notag
\end{align}
$\eb^{(i)}$ is the $i$-th coordinate vector. Then by the proof of Lemma \ref{omega_n}, we know that $\{\bar f_i\}_{i=1}^n\in \cV^{(1)}$, $\dist(\zero, \bar\xb^*) = 1$ where $\bar\xb^*$ is the global minimizer of $\bar F$ and for any $\xb$ satisfying $|\text{supp}\{\xb\}| \leq n/2$, 
\begin{align}
    \bar F(\xb) - \inf_{\xb} F(\xb) \geq 1/4.\label{omega_n_convex_2}
\end{align}
Next we choose $f_i = \lambda \bar f_i(\xb/\beta)$, where $\lambda = LB^2, \beta =B $, then we can check that $\{f_i\}_{i=1}^n \in \cV^{(L)}$, $\dist(\zero, \xb^*) = B$ where $\xb^*$ is the global minimizer of $F$, $F \in \cS^{(L,L)} \subset \cS^{(0,L)}$. Moreover, since $\nabla f_i(\xb) = -\lambda\sqrt{n}\eb^{(i)}/\beta + \lambda\xb/\beta^2$, then for some $\xb$, $i$ is in the support set of $\xb$ only if $f_i$ has been called. Thus, if less than $n/2$ IFO calls have been made, then current point $\xb$ satisfies that $|\text{supp}\{\xb\}| \leq n/2$. With \eqref{omega_n_1}, we have that
$F(\xb) - \inf_{\zb}F(\zb) \geq \lambda/4 \geq \epsilon $. 
\end{proof}

\bibliographystyle{ims}
\bibliography{reference}

\end{document}